\documentclass{article}

\usepackage[dvips]{graphicx}
\usepackage{amsfonts}

\usepackage{mathrsfs, amsmath, amssymb, amsthm, subfigure, bm}

\theoremstyle{plain}\newtheorem{theorem}{Theorem}
\newtheorem{lemma}[theorem]{Lemma}
\newtheorem{proposition}[theorem]{Proposition}
\newtheorem{corollary}[theorem]{Corollary}
\theoremstyle{definition}\newtheorem{definition}{Definition}
\theoremstyle{definition}\newtheorem{example}{Example}

\begin{document}

\title{FINITE SETS OF AFFINE POINTS WITH
 UNIQUE \\ ASSOCIATED MONOMIAL ORDER QUOTIENT BASES}

\author{ZHE LI, SHUGONG ZHANG, and TIAN DONG\\
\footnotesize School of Mathemathics, Key Lab.  of Symbolic Computation\\
\footnotesize  and Knowledge Engineering (Ministry of Education),\\
\footnotesize   Jilin University, Changchun 130012,  PR China
\footnote{lizhe200809@gmail.com,sgzh@jlu.edu.cn,dongtian@jlu.edu.cn}}

\date{}

\maketitle


\begin{abstract}
The quotient bases for zero-dimensional ideals are often of interest in the investigation of multivariate
polynomial interpolation,  algebraic coding theory, and computational
molecular biology, etc. In this paper, we discuss the properties of zero-dimensional ideals with unique monomial order quotient bases, and verify that the vanishing ideals of Cartesian sets have unique monomial order quotient bases.
Furthermore, we reveal the relation between Cartesian sets and the point sets with unique associated monomial order quotient bases.

Keywords: Monomial order quotient basis;  Cartesian set; Zero-dimensional ideal.

2000 Mathematics Subject Classification: 13P10
\end{abstract}

\section{Introduction}
\label{sec:int}

Let $\mathbb{F}$ be a field of characteristic zero, and suppose that $\mathbb{F}[\bm{x}]:=\mathbb{F}[x_1,\ldots,x_d]$
is the polynomial ring in $d$ variables over $\mathbb{F}$.
A finite linearly independent set $\Theta$ of linear functionals, mapping $\mathbb{F}[\bm{x}]$ to $\mathbb{F}$, is said to admit an \emph{ideal interpolation scheme} if
$$\ker\Theta=\{p \in \mathbb{F}[\bm{x}]: \theta(p)=0, \mbox{for~all~}\theta\in \Theta\}$$
forms a zero-dimensional ideal, cf.\cite{Birkhoff1979, GascaSauer2000, dBo2005, Shekhtman2009, LZD2011}. It was shown in \cite{deBoor1991, MMM1993} that this holds if and only if there exists a finite set $\Xi \subset \mathbb{F}^d$ and $D$-invariant finite-dimensional subspaces $Q_{\bm{\xi}}\subset \mathbb{F}[\bm{x}]$,
$\bm{\xi}\in \Xi$, such that
$$
\mathrm{span}_{\mathbb{F}}\Theta=\mathrm{span}_{\mathbb{F}}\{\delta_{\bm{\xi}}\circ q(D):q\in Q_{\bm{\xi}},\bm{\xi}\in \Xi\},
$$
where $D$ is the differentiation symbol and
$\delta_{\bm{\xi}}$ denotes the point evaluation functional at $\bm{\xi}$.

We now assume that $\ker\Theta$ forms a zero-dimensional ideal. For an order ideal $O\subset \mathbb{F}[\bm{x}]$,
if the set $\{\bm{t}+\ker\Theta: \bm{t}\in O\}$ provides an $\mathbb{F}$-vector space basis for
quotient ring $\mathbb{F}[\bm{x}]/\ker\Theta$,
then $O$ forms a \emph{monomial order quotient basis} for $\ker\Theta$.
In many applications such as multivariate polynomial interpolation\cite{Sauer1997}, algebraic coding theory\cite{Gei2009}, and computational molecular biology\cite{Lun2010}, the structure of monomial order quotient bases rather than Gr\"{o}bner bases is of particular interest.

The standard example of ideal interpolation is Lagrange interpolation. In this case, $\ker\Theta=\mathcal {I}(\Xi)$, the vanishing ideal of $\Xi$, and a monomial order quotient basis $O$ for $\mathcal {I}(\Xi)$ is also called an \emph{associated monomial order quotient basis} of $\Xi$. As is well known, the structure of $O$ depends not only on the cardinal but significantly on the geometry of $\Xi$, cf. \cite{SauerXu19951, Sauer2006}.
As interpolation point sets with special geometries, Cartesian sets in $\mathbb{F}^d$ have been studied by many authors \cite{Lor1992, GascaSauer2000:2, Cra2004, CDZ2006}.
In 2004, T. Sauer \cite{Sau2004} showed that the vanishing ideal of a Cartesian set has a unique Gr\"{o}bner \'{e}scalier, independent of the monomial order.
The natural questions are: (1) Does a Cartesian set also have a unique associated monomial order quotient basis?
(2) If so, is there any non-Cartesian set satisfying this property?

In this paper, we first introduce two criteria for determining whether a zero-dimensional ideal has a unique monomial order quotient basis or not. With the aid of them, we show that every Cartesian set has
a unique associated monomial order quotient basis, and furthermore that for $d\geq 3$, there always exists at least one non-Cartesian point set that also has this property. The main results of this paper will be put in Section 3.
The next section, Section \ref{sec:pre}, will give some notation and background results.

\section{Notation and Background Results}\label{sec:pre}

In this section, we will settle the key notation used throughout the paper and give some background results. For more details,  we refer the reader to \cite{Lor1992, Buc2006, Buc1970, Buc1985, KR2005}.

We use $\mathbb{N}_0$ to stand for the monoid of
nonnegative integers and boldface type for tuples with their entries
denoted by the same letter with subscripts, for example,
$\bm{\alpha}=(\alpha_1,\ldots, \alpha_d)$.

Henceforward, $\leq$ will denote the usual product order on
$\mathbb{N}_0^d$, that is, for arbitrary
$\bm{\alpha}$, $\bm{\beta}\in\mathbb{N}_0^d$, $\bm{\alpha}\leq \bm{\beta}$ if and only if $\alpha_i\leq \beta_i, i=1,\ldots, d$.
A finite subset $\mathcal {A}\subset \mathbb{N}_0^d$ is \emph{lower} if
for every $\bm{\alpha}\in \mathcal{A}$, $\bm{0}\leq\bm{\beta}\leq \bm{\alpha}$ implies  $\bm{\beta}\in \mathcal{A}$.

A \emph{monomial} ${\bm{x}}^{\bm{\alpha}}\in \mathbb{F}[\bm{x}]$ is
a power product of the form $x_1^{\alpha_1}\cdots x_d^{\alpha_d}$
with $\bm{\alpha}\in \mathbb{N}_0^d$.  Denote by
$\mathbb{T}(\bm{x}):=\mathbb{T}(x_1, \ldots, x_d)$ the monoid of
all monomials in $\mathbb{F}[\bm{x}]$.
A finite monomial set $O\subset \mathbb{T}(\bm{x})$ is called an
\emph{order ideal} if for every $\bm{t}\in O$, $\bm{t'}|\bm{t}$ implies $\bm{t'}\in O$.
Further, the \emph{corner} of an order ideal $O$ is
\begin{equation}\label{corner}
    \mathcal{C}[O]=\{\bm{t}\in \mathbb{T}(\bm{x}): \bm{t}\notin O, x_i|\bm{t}\Rightarrow
\bm{t}/x_i\in O, 1\leq i\leq d\}.
\end{equation}

For each fixed monomial order $\prec$ on $\mathbb{T}({\bm{x}})$, a
nonzero polynomial $f \in \mathbb{F}[{\bm{x}}]$ has a unique \emph{leading
monomial} $\mathrm{LM}_{\prec}(f)$, which is the greatest one
appearing in $f$ w.r.t. $\prec$. For arbitrary $S \subset \mathbb{F}[\bm{x}]$,
the set of leading monomials of all polynomials in $S$ is denoted by $\mathrm{LM}_{\prec}(S)$.
According to \cite{Mor2009}, the monomial set
$$
\mathcal{N}_\prec(\mathcal{I}):=\mathbb{T}({\bm{x}})\backslash \mathrm{LM}_{\prec}(\mathcal{I})=\{{\bm{x}}^{\bm{\alpha}}\in \mathbb{T}({\bm{x}}):
\mathrm{LM}_{\prec}(f)\nmid {\bm{x}}^{\bm{\alpha}}, \forall f\in \mathcal{I}\}
$$
is
the \emph{Gr\"{o}bner \'{e}scalier} of ideal $\mathcal{I}$ w.r.t. $\prec$. In fact,
every Gr\"{o}bner \'{e}scalier of $\mathcal{I}$ forms a monomial order quotient basis for $\mathcal{I}$.

\begin{definition}\cite{Sau2004, CDZ2006}\label{Cartesiandefinition}
A finite set
$\Xi\subset \mathbb{F}^d$ of distinct points is \emph{Cartesian} if and only if there exists
a lower set $\mathcal{A}\subset \mathbb{N}_0^d$ and injective functions $y_i:\mathbb{N}_0\rightarrow \mathbb{F}, 1\leq i\leq d$, such that $\Xi$ can be written as
\begin{equation}\label{Cartesian}
\Xi=\left\{\left(y_1(\alpha_1), \ldots ,y_d(\alpha_d)\right)\in \mathbb{F}^d: \bm{\alpha}=(\alpha_1, \ldots, \alpha_d)
\in \mathcal {A}\right\}.
\end{equation}
$\Xi$ is also called an $\mathcal{A}$-\emph{Cartesian} set.
\end{definition}
For instance, point set
$\Xi=\{(2.3, 1.2), (4.7, 1.2), (1.5, 1.2), (2.3, 0.2)\}\subset \mathbb{Q}^2$ is a $\{(0,0), (1,0), (2,0), (0,1)\}$-Cartesian set.

\begin{theorem}\textup{\cite{Sau2004}}\label{Sauer}
Let $\Xi\subset \mathbb{F}^d$ be an $\mathcal{A}$-Cartesian set, then
Gr\"{o}bner \'{e}scalier $\mathcal{N}_{\prec}(\mathcal {I}(\Xi))$
w.r.t. any monomial order $\prec$ is identical to
\begin{equation}\label{carmonbas1}
\mathfrak{N}:=\{{\bm{x}}^{\bm{\alpha}}:\bm{\alpha}\in
\mathcal {A}\}.
\end{equation}
\end{theorem}

\section{Main Results}\label{sec:main}

\subsection{Zero-dimensional ideals with unique monomial order quotient bases}

In this subsection we will establish two criteria for zero-dimensional ideals with unique monomial order quotient bases. First, we need the following easy lemma.

\begin{lemma}\label{gctest2l}
Suppose that the monomials in the finite set $T\subset\mathbb{T}(\bm{x})$ are linearly
independent modulo an ideal $\mathcal {I}\subset \mathbb{F}[\bm{x}]$.
For each  $1\leq i \leq m$, let $T_i$ be a subset of $T$ and let $V_i$ be the $\mathbb{F}$-vector space generated by $\{\bm{t}+\mathcal {I}:\bm{t} \in T_i\}$. Then
$\bigcap_{i=1}^m V_i$ is generated by  $\{\bm{t}+\mathcal {I}: \bm{t}\in \bigcap_{i=1}^m T_i\}$.
\end{lemma}

\begin{proof} We will use induction on $m$, the number of the monomial sets, to prove the lemma.
When $m=2$, for an arbitrary $\bm{u}+\mathcal{I}\in V_1\bigcap V_2$, there exist $k_{\bm{t}}, l_{\bm{t}}, k_{\bm{t}'}, l_{\bm{t}''}\in \mathbb{F}$ such that
$$
\left(\sum_{\bm{t}\in  T_1\cap T_2} k_{\bm{t}} \bm{t}+\sum_{\bm{t}'\in T_1\setminus T_2}k_{\bm{t}'}\bm{t}'\right) ~\mbox{mod}~\mathcal {I}\equiv
\bm{u}\equiv \left(\sum_{\bm{t}\in  T_1\cap T_2} l_{\bm{t}} \bm{t}+\sum_{\bm{t}''\in T_2\setminus T_1}l_{\bm{t}''} \bm{t}''\right)  ~\mbox{mod}~\mathcal {I}
$$
which implies that
$$\left(\sum_{\bm{t}\in  T_1\cap T_2} (k_{\bm{t}}-l_{\bm{t}})\bm{t}+\sum_{\bm{t}'\in T_1\setminus T_2}k_{\bm{t}'}\bm{t}'-\sum_{\bm{t}''\in T_2\setminus T_1}l_{\bm{t}''} \bm{t}''\right)\equiv 0 ~\mbox{mod}~\mathcal {I}.$$
Consequently, we have
$$\bm{u}\equiv \sum_{\bm{t}\in  T_1\cap T_2} k_{\bm{t}} \bm{t}~\mbox{mod}~\mathcal {I}.$$
Now, assume that the lemma is true for $m-1$, i.e.,
$\bigcap_{i=1}^{m-1} V_i$ is generated by  $\{\bm{t}+\mathcal {I}: \bm{t}\in \bigcap_{i=1}^{m-1} T_i\}$.
Since $V_m$ is generated by  $\{\bm{t}+\mathcal {I}: \bm{t}\in T_m\}$ and $\bigcap_{i=1}^{m-1} T_i\subset T$, the statement for $m=2$ implies the lemma.
\end{proof}

The following proposition presents the first criterion for zero-dimensional ideals
with unique monomial order quotient bases whose proof uses the notion of elimination order. Actually,
a monomial order on
$\mathbb{T}(\bm{x})$ is called an \emph{elimination order for} $x_i$ if the
monomial $x_i$ is greater than all monomials in
$$
\mathbb{T}(\bm{x}\setminus
x_i):=\mathbb{T}(x_1, \ldots, x_{i-1}, x_{i+1}, \ldots, x_d)
$$
w.r.t. this order. A typical example is the lexicographic order $\succ_{lex(i)}$ with
$$
x_i\succ_{lex(i)}\cdots \succ_{lex(i)} x_d \succ_{lex(i)} x_1 \succ_{lex(i)} \cdots \succ_{lex(i)}x_{i-1}.
$$

\begin{proposition}\label{gctest11}
Let $\mathcal {I}\subset \mathbb{F}[\bm{x}]$ be a zero-dimensional ideal and let $O$ be a monomial order quotient basis for $\mathcal {I}$.
Then this monomial order quotient basis is unique
if and only if for each ${\bm{x}}^{\bm{\alpha}} \in
\mathcal {C}[O]$, the set $\{\bm{x}^{\bm{\beta}}: 0\leq \bm{\beta}\leq\bm{\alpha}\}$ is linearly dependent modulo $\mathcal {I}$.
\end{proposition}

\begin{proof}
$\Longrightarrow$:
For each $1\leq i\leq d$, let $\prec_i$ be an elimination order for $x_i$. Then by the uniqueness of the monomial order quotient basis for $\mathcal {I}$, we have
\begin{equation}\label{O}
\mathcal{N}_{\prec_1}(\mathcal {I})=\cdots=\mathcal{N}_{\prec_d}(\mathcal {I})=O.
\end{equation}
Next, we will show that for each ${\bm{x}}^{\bm{\alpha}} \in \mathcal {C}[O]$,
$\{\bm{x}^{\bm{\beta}}: 0\leq \bm{\beta}\leq\bm{\alpha}\}$ is linearly dependent modulo $\mathcal {I}$.

An arbitrary ${\bm{x}}^{\bm{\alpha}}$ in $\mathcal{C}[O]$ gives rise to monomial sets
$$O_i:=\{{\bm{x}}^{\bm{\beta}}\in O:  {\bm{x}}^{\bm{\beta}} \prec_i {\bm{x}}^{\bm{\alpha}}\}, \quad 1\leq i\leq d.$$
For each $1\leq i\leq d$,
let $V_i$ be the $\mathbb{F}$-vector space generated by $\{\bm{t}+\mathcal {I}: \bm{t}\in O_i\}$. It follows from the definition of the elimination order for $x_i$ and (\ref{corner}) that
\begin{equation}\label{intersect}
\bigcap_{i=1}^d O_i \subseteq \{{\bm{x}}^{\bm{\beta}}: 0\leq \bm{\beta}<\bm{\alpha}\}\subseteq O.
\end{equation}
For each $1\leq i\leq d$, ${\bm{x}}^{\bm{\alpha}}\in
\mathcal{C}[O]=\mathcal{C}[\mathcal{N}_{\prec_{i}}(\mathcal {I})]$ implies that the set
$ O_i\bigcup\{{\bm{x}}^{\bm{\alpha}}\}$ is linearly dependent modulo $\mathcal {I}$. That is to say,
${\bm{x}}^{\bm{\alpha}}+\mathcal{I}\in V_i, 1\leq i\leq d$. Therefore, ${\bm{x}}^{\bm{\alpha}}+\mathcal{I}\in\bigcap_{i=1}^n V_i$.
Recalling Lemma \ref{gctest2l}, we know that
$\bigcap_{i=1}^n V_i$ is generated by $\{\bm{t}+\mathcal {I}: \bm{t}\in\bigcap_{i=1}^n O_i\}$. Hence,
$(\bigcap_{i=1}^n O_i)\bigcup \{{\bm{x}}^{\bm{\alpha}}\}$ is linearly dependent modulo $\mathcal {I}$.
According to (\ref{intersect}), $\{{\bm{x}}^{\bm{\beta}}: 0\leq \bm{\beta}\leq\bm{\alpha}\}$
is linearly dependent module  $\mathcal {I}$.

$\Longleftarrow$: Suppose that there exists another monomial order quotient basis
$O'\neq O$ for $\mathcal {I}$. Let ${\bm{x}}^{\alpha}$ be an arbitrary
element in $O'\setminus O$. We now have two cases to
consider.

Case I: ${\bm{x}}^{\bm{\alpha}} \in  O'\cap\mathcal{C}[O]$.

In this case, $\left\{\bm{x}^{\bm{\beta}}:0\leq \bm{\beta}\leq \bm{\alpha}\right\}$ is linearly dependent modulo $\mathcal {I}$.
${\bm{x}}^{\bm{\alpha}}\in O'$ leads to $\left\{\bm{x}^{\bm{\beta}}:0\leq \bm{\beta}\leq \bm{\alpha}\right\}\subseteq O'$ since $O'$ is an order ideal, which implies that
$O'$ is linearly dependent modulo $\mathcal {I}$. This is a
contradiction to our hypothesis that $O'$ is a monomial order quotient basis for $\mathcal {I}$. Hence, ${\bm{x}}^{\bm{\alpha}} \notin
\mathcal{C}[O]$.

Case II: ${\bm{x}}^{\bm{\alpha}} \in O'\setminus \mathcal{C}[O]$.

In this case, ${\bm{x}}^{\bm{\alpha}}\notin O \cup \mathcal{C}[O]$.  It follows from the definition of $\mathcal{C}[O]$ that there exists ${\bm{x}}^{\bm{\gamma}}\in \mathcal{C}[O]$ such that
${\bm{x}}^{\bm{\gamma}}|{\bm{x}}^{\bm{\alpha}}$, which implies that ${\bm{x}}^{\bm{\gamma}}\in
O'$, and hence $\{\bm{x}^{\bm{\beta}}:0\leq \bm{\beta}\leq \bm{\gamma}\}\subseteq O'$. By our assumption,
$\{\bm{x}^{\bm{\beta}}:0\leq \bm{\beta}\leq \bm{\gamma}\}$ is linearly dependent modulo $\mathcal {I}$, and
consequently $O'$ is linearly dependent modulo $\mathcal {I}$ which leads to a contradiction too. Therefore, we conclude that ${\bm{x}}^{\bm{\alpha}}\notin O'\setminus \mathcal{C}[O]$.

In sum, we can deduce that ${\bm{x}}^{\bm{\alpha}}\notin
O'$, which contradicts the hypothesis ${\bm{x}}^{\bm{\alpha}}\in O'\setminus O$.
\end{proof}

As the main theorem of this paper,  Theorem \ref{gctest2} characterizes a new criterion for zero-dimensional ideals with unique monomial order quotient bases by elimination orders.

\begin{theorem}\label{gctest2}
For each $1\leq i\leq d$, let $\prec_i$ be an elimination order for $x_i$. Then for each zero-dimensional ideal $\mathcal {I}$,
the following are equivalent:
\begin{enumerate}
  \item[\textup{(1)}] $\mathcal {I}$ has a unique monomial order quotient basis.
  \item[\textup{(2)}] For any two monomial orders $\prec$ and $\prec'$, $\mathcal{N}_{\prec}(\mathcal {I})=\mathcal{N}_{\prec'}(\mathcal {I})$ .
  \item[\textup{(3)}] $\mathcal{N}_{\prec_i}(\mathcal {I}), 1\leq i\leq d$, are identical.
\end{enumerate}
\end{theorem}

\begin{proof}
(1)$\Rightarrow$(2): Since $\mathcal {I}$ has a unique monomial order quotient basis,
$\mathcal{N}_{\prec}(\mathcal {I})$ and $\mathcal{N}_{\prec'}(\mathcal {I})$ obviously coincide for any two monomial orders $\prec$ and $\prec'$.

(2)$\Rightarrow$(3) is trivial.

(3)$\Rightarrow$(1): Set
$$
O:=\mathcal{N}_{\prec_1}(\mathcal {I})=\cdots=\mathcal{N}_{\prec_d}(\mathcal {I}).
$$
The arguments in the proof of Proposition \ref{gctest11} implies that for each ${\bm{x}}^{\bm{\alpha}}\in\mathcal{C}[O]$,
$\{{\bm{x}}^{\bm{\beta}}: 0\leq \bm{\beta}\leq\bm{\alpha}\}$
is linearly dependent module  $\mathcal {I}$. Therefore,  $O$ is the unique monomial order quotient basis for $\mathcal {I}$.
\end{proof}

\begin{lemma}
Let $\mathcal {I}\subset \mathbb{F}[\bm{x}]$ be a zero-dimensional ideal,
and let $G$ and $G'$ be the reduced Gr\"{o}bner bases for $\mathcal{I}$ w.r.t.
monomial orders $\prec$ and $\prec'$ respectively.
Then $\mathcal{N}_{\prec}(\mathcal {I})=\mathcal{N}_{\prec'}(\mathcal {I})$ if and only if
$\mathrm{LM}_{\prec}(G)=\mathrm{LM}_{\prec'}(G')$. Furthermore, if $\mathrm{LM}_{\prec}(G)=\mathrm{LM}_{\prec'}(G')$, then
$G=G'$.
\end{lemma}

\begin{proof}
It is easy to see that $\mathcal {C}[\mathcal{N}_{\prec}(\mathcal {I})]=\mathrm{LM}_{\prec}(G)$.
Then $\mathcal{N}_{\prec}(\mathcal {I})=\mathcal{N}_{\prec'}(\mathcal {I})$ if and only if $\mathrm{LM}_{\prec}(G)=\mathrm{LM}_{\prec'}(G')$.
Moreover, assume that $\mathrm{LM}_{\prec}(G)=\mathrm{LM}_{\prec'}(G')$. Then the first statement of this lemma implies $\mathcal{N}_{\prec}(\mathcal {I})=\mathcal{N}_{\prec'}(\mathcal {I}):=O$. Therefore,
for each pair $g\in G$, $g'\in G'$ satisfying $\mathrm{LM}_{\prec}(g)=\mathrm{LM}_{\prec'}(g')$,
$g-g'\in \mathrm{span}_{\mathbb{F}}O\cap \mathcal {I}=\{0\}$ follows by the property of reduced Gr\"{o}bner bases, which leads to the lemma.
\end{proof}

The following corollary follows immediately from Theorem 3.1 and Lemma 3.2.

\begin{corollary}
Let $\mathcal {I}\subset \mathbb{F}[\bm{x}]$ be a zero-dimensional ideal and $G_i$ be the reduced Gr\"{o}bner basis for $\mathcal {I}$ w.r.t.
an elimination order $\prec_i$ for $x_i$.  If $\mathrm{LM}_{\prec_i}(G_i)$, $1\leq i\leq d$, are identical,
then $\mathcal {I}$ has a unique reduced Gr\"{o}bner basis, independent of the monomial order.
\end{corollary}

Following T. Sauer \cite{Sau2004}, we refer to a reduced Gr\"{o}bner basis $G$ for an ideal $\mathcal{I}$
as a \emph{universal Gr\"{o}bner basis} if  $G$ is the unique reduced Gr\"{o}bner basis for $\mathcal{I}$, independent of the monomial order.

Suppose that $\Theta$ gives an ideal interpolation scheme. We can compute $\mathcal{N}_{\prec_i}(\ker\Theta)$, $1\leq i\leq d$, to decide whether $\ker\Theta$ has a unique monomial order quotient basis or not.
If
\begin{equation}\label{specialinterpolationfunctionals}
\Theta=\left\{\delta_{\bm{\xi}_{k}}\circ\frac{\partial^{\|\bm{\alpha}\|}}{\partial x_1^{\alpha_1}\ldots \partial x_d^{\alpha_d}}:
{\bm{\alpha}}\in \mathcal {A}_{k},  1\leq
k\leq \mu\right\},
\end{equation}
where $\bm{\xi}_{1},\ldots,\bm{\xi}_{\mu}\in \mathbb{F}^d$
are distinct and $\mathcal {A}_{1},\ldots,\mathcal
\mathcal {A}_{\mu}\subset \mathbb{N}_0^d$ are lower, then
we can use algorithms, such as the ones in \cite{LM1995, FRR2006}, to
compare $\mathcal{N}_{\prec_{lex(i)}}(\ker\Theta)$ without computing the Gr\"{o}bner bases.

\begin{example}
Let
\begin{align*}
\Theta=\Big\{&\delta_{(0,0,0)}, \delta_{(0,0,0)}\circ \frac{\partial }{\partial x_3},
\delta_{(0,0,0)}\circ \frac{\partial }{\partial x_2},  \delta_{(0,0,0)}\circ \frac{\partial }{\partial x_1}, \\
&\delta_{(0,0,0)}\circ \frac{\partial^2
 }{\partial x_1 \partial x_3},
 \delta_{(1,0,0)}, \delta_{(1,0,0)}\circ  \frac{\partial }{\partial x_3}, \delta_{(0,0,1)}\Big\}.
 \end{align*}
A direct computation shows that
$$\mathcal{N}_{\prec_{lex(i)}}(\ker\Theta)=\{1,x_1,x_1^2,x_2,x_3,x_3^2,x_1
x_3,x_1^2 x_3\},\quad i=1,2,3,$$
which concludes that $\ker\Theta$ has a unique monomial order quotient basis.
\end{example}

\subsection{Vanishing ideals of Cartesian sets}

From Theorem \ref{Sauer} we see that the vanishing ideal of a Cartesian set has a
unique Gr\"{o}bner \'{e}scalier, independent of the monomial order. However, there also exists certain associated monomial order quotient basis of some point set which can not be Gr\"{o}bner \'{e}scalier w.r.t. any monomial order.

For example, let $\Xi=\{(0, 0), (1.1, -0.1), (0.1, 0.9), (1, 1)\}$ be a point set in $\mathbb{Q}^2$. Then
$\{1, x_1, x_2, x_1 x_2\}$ forms a monomial order quotient basis but not a Gr\"{o}bner \'{e}scalier (w.r.t. any monomial order) of $\mathcal {I}(\Xi)$. Indeed, for an arbitrary
monomial order $\prec$, we have either $x_1^2\prec x_1 x_2$ or $x_2^2\prec x_1 x_2$.
Since both $\{1, x_1, x_2, x_1^2\}$ and $\{1, x_1, x_2, x_2^2\}$ are linearly independent modulo $\mathcal {I}(\Xi)$,  $x_1^2$
or $x_2^2$
must belong to $\mathcal {N}_{\prec}(\mathcal {I}(\Xi))$.

Next, we will show that Cartesian sets
have unique associated monomial order quotient bases.  To achieve this, we need to propose a criterion for multidimensional Cartesian sets, which
is the extension of Theorem 5 of \cite{CDZ2006} to general dimension.

Specifically, for fixed $1\leq i \leq d$, we cover a finite point set $\Xi\subset \mathbb{F}^d$ with exactly $m_i+1, m_i\in \mathbb{N}$, affine hyperplanes $l_{i, 0},  l_{i, 1}, \ldots,  l_{i, m_i}\subset \mathbb{F}^d$ perpendicular to the $x_i$-axis.
Moreover, set
\begin{align}\label{cover}
  \mathscr{L}_{i, j}(\Xi):=\{&(\xi_1, \ldots, \xi_{i-1}, \xi_{i+1}, \ldots,\xi_{d}):\nonumber\\
&\forall (\xi_1, \ldots, \xi_{d})\in \Xi \cap l_{i, j}\}\subset \mathbb{F}^{d-1},\quad 1\leq i\leq d, 0\leq j\leq m_i.
\end{align}
Without loss of generality, we assume that
$\#\mathscr{L}_{i,0}({\Xi})\geq\cdots\geq\#\mathscr{L}_{i, m_i}({\Xi})$. 

\begin{theorem}\label{3dCartesiantest}
A finite set of distinct points $\Xi\subset \mathbb{F}^d, d\geq 2$, is
\emph{Cartesian} if and only if for each $1\leq i\leq d$, we have
\begin{equation}\label{3dCartesianteste}
\mathscr{L}_{i,0}(\Xi)\supseteq \cdots \supseteq
\mathscr{L}_{i,m_i}(\Xi),
\end{equation}
where $m_i$ and $\mathscr{L}_{i,j}(\Xi)$ are as above.
\end{theorem}

\begin{proof}
$\Longrightarrow$: Recall (\ref{Cartesian}). A Cartesian set $\Xi\subset \mathbb{F}^d$  can be written as
$$
\Xi=\left\{\left(y_1(\alpha_1), \ldots ,y_d(\alpha_d)\right)\in \mathbb{F}^d: \bm{\alpha}=(\alpha_1, \ldots, \alpha_d)
\in \mathcal {A}\right\}
$$
with lower set $\mathcal{A}\subset \mathbb{N}_0^d$ and injective functions $y_i:\mathbb{N}_0\rightarrow \mathbb{F}, 1\leq i\leq d$.

Fix $1\leq i\leq d$. For any $0\leq \gamma_i\leq\beta_i\leq
m_i$, according to (\ref{cover}), we have
\begin{align}
    \mathscr{L}_{i,\beta_i}(\Xi)=\Big\{&\Big(y_1(\alpha_1), \ldots, y_{i-1}(\alpha_{i-1}), y_{i+1}(\alpha_{i+1}), \ldots, y_d(\alpha_d)\Big):\nonumber\\
   &(\alpha_1, \ldots,\alpha_{i-1}, \beta_i,
\alpha_{i+1},\ldots,\alpha_{d}) \in \mathcal {A}\Big\},\label{carproof}\\
    \mathscr{L}_{i,\gamma_i}(\Xi)=\Big\{&\Big(y_1(\alpha_1), \ldots, y_{i-1}(\alpha_{i-1}), y_{i+1}(\alpha_{i+1}), \ldots, y_d(\alpha_d)\Big):\nonumber\\
    &(\alpha_1, \ldots,\alpha_{i-1}, \gamma_i,
\alpha_{i+1},\ldots,\alpha_{d}) \in \mathcal {A}\Big\}\label{carproof2},
\end{align}
which imply immediately that $\mathscr{L}_{i, \beta_i}(\Xi)\subseteq \mathscr{L}_{i, \gamma_i}(\Xi)$ for each $1\leq i\leq d$, $0\leq \gamma_i\leq\beta_i\leq
m_i$.

$\Longleftarrow$: Fix a point $\bm{\xi}=(\xi_1, \ldots, \xi_d)$ in $\Xi$. Let $l_{i, j}, 1\leq i\leq d, 0\leq
j\leq m_i, $ be covering hyperplanes defining $\mathscr{L}_{i,j}(\Xi)$. Thus, there must exist $\bm{\alpha}\in \mathbb{N}_0^d$ such that
$\{\bm{\xi}\}=\bigcap_{i=1}^d l_{i, \alpha_i}$, where  $0\leq
\alpha_{i}\leq m_i$ for all $1\leq i\leq d$.
Set
$$
\mathcal{A}:=\left\{\bm{\alpha}\in \mathbb{N}_0^d: \bigcap_{i=1}^d
l_{i, \alpha_i} \in \Xi\right\}.
$$
It is easy to observe that $\Xi$ can be described as
$$\Xi=\{\left(y_1(\alpha_1), \ldots ,y_d(\alpha_d)\right): (\alpha_1, \ldots, \alpha_d)\in \mathcal{A}\},$$
where $y_i:\mathbb{N}_0\rightarrow \mathbb{F}, 1\leq i\leq d$, are certain injective functions.
Let $(\alpha_1, \ldots, \alpha_d) \in \mathcal {A}$. For any
$(\beta_1, \ldots, \beta_d)\leq (\alpha_1, \ldots, \alpha_d)$ in $\mathbb{N}_0^d$,
by (\ref{3dCartesianteste}), we have
$$ (y_1(\alpha_1), \ldots ,y_{d-1}(\alpha_{d-1}))\in \mathscr{L}_{d, \alpha_d}(\Xi)\subseteq \mathscr{L}_{d, \beta_d}(\Xi).$$
Thus,
$
(\alpha_1, \ldots, \alpha_{d-1}, \beta_d)\in \mathcal{A}.
$
Moreover,  $$(y_1(\alpha_1), \ldots,
y_{d-2}(\alpha_{d-2}),
y_{d}(\beta_{d}))\in \mathscr{L}_{d-1, \alpha_{d-1}}(\Xi)\subseteq \mathscr{L}_{d-1, \beta_{d-1}}(\Xi)$$
implies $(\alpha_1, \ldots, \alpha_{d-2}, \beta_{d-1}, \beta_d)\in \mathcal{A}$, and so on. Finally, we can deduce that
$(\beta_1, \ldots, \beta_d)\in \mathcal{A}$, which means that $\mathcal{A}$ is lower.
\end{proof}

The next lemma and corollary establish a relation between a Cartesian set $\Xi\subset \mathbb{F}^d$ and its associated point sets $\mathscr{L}_{i, j}(\Xi)\subset\mathbb{F}^{d-1}$.
 \begin{lemma}\label{miancar}
Let $\Xi \subset \mathbb{F}^d$, $d\geq 2$, be  an $\mathcal{A}$-Cartesian set, and let
$m_i$, $\mathscr{L}_{i, j}(\Xi)$ be as in \textup{(\ref{3dCartesianteste})}. For $1\leq i\leq d, 0\leq j\leq m_i$, set
\begin{align}
    \mathcal{A}_{i,j}:=\{&(\alpha_1, \ldots, \alpha_{i-1}, \alpha_{i+1}, \ldots, \alpha_d):\nonumber \\
    &(\alpha_1, \ldots, \alpha_{i-1}, j, \alpha_{i+1}, \ldots, \alpha_d)\in\mathcal
{A} \}.\label{Aij}
\end{align}
Then $\mathscr{L}_{i,j}(\Xi)$ is $(d-1)$-dimensional
$\mathcal{A}_{i,j}$-Cartesian for all $1\leq i\leq d, 0\leq j\leq m_i$.
\end{lemma}
\begin{proof}
Fix $1\leq i\leq d$, $0\leq j\leq m_i$.
For each $(\alpha_1, \ldots, \alpha_{i-1}, \alpha_{i+1}, \ldots, \alpha_d)\in \mathcal{A}_{i,j}$,
$\bm{0}\leq(\beta_1, \ldots, \beta_{i-1}, \beta_{i+1}, \ldots, \beta_d)\leq (\alpha_1, \ldots, \alpha_{i-1}, \alpha_{i+1}, \ldots, \alpha_d)$
implies that
$(\beta_1, \ldots, \beta_{i-1}, j, \beta_{i+1}, \ldots, \beta_d)\in \mathcal{A}$. Then
$(\beta_1, \ldots, \beta_{i-1}, \beta_{i+1}, \ldots, \beta_d)\in \mathcal{A}_{i,j}$ follows. Thus, $\mathcal{A}_{i,j}\subset \mathbb{N}_0^{d-1}$ is lower.
By (\ref{carproof}), (\ref{carproof2}) and (\ref{Aij}),  we can conclude that  $\mathscr{L}_{i,j}(\Xi)$ is
$\mathcal{A}_{i,j}$-Cartesian.

\end{proof}

\begin{corollary}
Let $\Xi$, $\mathcal{A}$ and $\mathcal{A}_{i,j}$ be as in \emph{Lemma \ref{miancar}}. Define
\begin{align*}
    \mathcal{A}_{i,j}\oplus
j:=\{&(\alpha_1, \ldots, \alpha_{i-1}, j, \alpha_{i+1}, \ldots, \alpha_d):\\
&\forall
(\alpha_1, \ldots, \alpha_{i-1}, \alpha_{i+1}, \ldots, \alpha_d)\in
\mathcal{A}_{i,j}\}.
\end{align*}
Then for each $1\leq i\leq d$, we have
$$
\mathcal{A}=\bigcup\limits_{j=0}^{m_i}\mathcal{A}_{i,j}\oplus
j.
$$
\end{corollary}

\begin{proposition}\label{carquot}
Let $\Xi\subset \mathbb{F}^d$ be an $\mathcal{A}$-Cartesian set, and suppose that $\mathfrak{N}$
is as in \textup{(\ref{carmonbas1})}. Then $\mathcal {I}(\Xi)$ has a unique monomial order quotient basis $\mathfrak{N}$.
\end{proposition}

\begin{proof}
We will prove this proposition by induction on $d$. When $d=1$, the proposition is obviously true.
Suppose that for any $(k-1)$-dimensional Cartesian set, the statement holds.

Let $\Xi\subset \mathbb{F}^k$ be an $\mathcal{A}$-Cartesian set where $\mathcal{A}\subset \mathbb{N}_0^k$, and
assume that $m_i$, $\mathscr{L}_{i, j}(\Xi)$ are as in \textup{(\ref{3dCartesianteste})}. For an arbitrary $\bm{x}^{\bm{\gamma}}\in \mathcal{C}[\mathfrak{N}]$,
there are two cases which must be examined.

Case I: $\bm{\gamma}=(0, \ldots, 0, m_k+1)$.

Let $g(x_k)=\prod_{j=0}^{m_k}l_{k, j}$.
Since
$l_{k,0}, l_{k,1}, \ldots, l_{k, m_k}$ are affine hyperplanes
covering $\Xi$, then $g(x_k)\in \mathcal {I}(\Xi)$ follows.
Furthermore,
it follows from the arguments in Theorem \ref{3dCartesiantest} that
\begin{equation}\label{ld}
l_{k,j}=x_k-y_k(j), \quad j=0,\ldots,m_k,
\end{equation}
which implies that  $\{{\bm{x}}^{\bm{\beta}}: 0\leq \bm{\beta}\leq \bm{\gamma}\}$ is linearly dependent modulo $\mathcal {I}(\Xi)$.

Case II: $\bm{\gamma}=(\gamma_1, \ldots, \gamma_{k-1}, j)$ for some $0\leq j\leq
m_k$, and $x_1^{\gamma_1}\ldots x_{k-1}^{\gamma_{k-1}}\in \mathcal
{C}[\mathcal{N}_{k,j}]$ where
$$
\mathcal{N}_{k,j}=\{x_1^{\alpha_1}\cdots
x_{k-1}^{\alpha_{k-1}}:(\alpha_1, \ldots, \alpha_{k-1})\in \mathcal
{A}_{k,j}\}.
$$

By inductive
hypothesis, $\mathcal{N}_{k,j}$ is the unique monomial order quotient basis for
$\mathcal {I}(\mathscr{L}_{k,j}(\Xi))$, then there must exist $c_{(\alpha_1, \ldots, \alpha_{k-1})}\in \mathbb{F}$ such that polynomial
\begin{align*}
    \hat{g}(x_1, \ldots, x_{k-1})=&x_1^{\gamma_1}\cdots
x_{k-1}^{\gamma_{k-1}}+\\
   &\sum\limits_{(\alpha_1, \ldots, \alpha_{k-1})<(\gamma_1, \ldots, \gamma_{k-1})}
c_{(\alpha_1, \ldots, \alpha_{k-1})}x_1^{\alpha_1}\cdots
x_{k-1}^{\alpha_{k-1}}\in \mathcal
{I}(\mathscr{L}_{k,j}(\Xi))
\end{align*}
by Proposition \ref{gctest11}. Set
$$
g(\bm{x}):=\hat{g}(x_1, \ldots, x_{k-1})\prod_{i=0}^{j-1}l_{k,i},
$$
where the empty product is defined to be 1.
Since
$$
\mathscr{L}_{k,0}(\Xi)\supseteq \cdots \supseteq
\mathscr{L}_{k,m_k}(\Xi),
$$
it is easy to check
that $\hat{g}(x_1, \ldots, x_{k-1})$
vanishes at $\Xi \setminus \{\Xi\cap\{\cup_{i=0}^{j-1}
l_{k, i}\}\}$ and $\prod_{i=0}^{j-1}l_{k,i}$
vanishes at $\Xi \cap\{\cup_{i=0}^{j-1}
l_{k,i}\}$. Therefore, $g(\bm{x})\in\mathcal
{I}(\Xi)$. It follows that $\{{\bm{x}}^{\bm{\beta}}: 0\leq \bm{\beta}\leq \bm{\gamma}\}$ is linearly dependent modulo $\mathcal {I}(\Xi)$.

In short, from Proposition
\ref{gctest11}, we conclude that $\mathfrak{N}$ is the unique monomial order quotient basis for $\mathcal{I}(\Xi)$.

\end{proof}

Finally, the following two theorems reveal the relation between Cartesian sets and the point sets with unique associated monomial order quotient bases.
\begin{theorem}\label{2dCartesiantest}
A finite set of distinct points $\Xi\subset \mathbb{F}^2$ is Cartesian if
and only if it has a unique associated monomial order quotient basis.
\end{theorem}

\begin{proof}
Lemma 2.1 of \cite{Cra2004} and
Theorem 6 of \cite{WZD2010} implies this theorem immediately.
\end{proof}

\begin{theorem}\label{relation}
For every $d\geq 3$, there exists at least one non-Cartesian point set in $\mathbb{F}^d$ which has a unique associated monomial order quotient basis.
\end{theorem}

\begin{proof}
Let
\begin{align*}
    \Xi_{d}=\{(0,0,0,\underset{d-3}{\underbrace{{0,\ldots,0}}}),(0,1,0,\underset{d-3}{\underbrace{{0,\ldots,0}}}),
    (0,0,1,\underset{d-3}{\underbrace{{0,\ldots,0}}}),(1,0,1,\underset{d-3}{\underbrace{{0,\ldots,0}}})\}
\end{align*}
with $d\geq 3$.
It is easy to check that $\{x_1(x_1-1), x_2(x_2-1), x_3(x_3-1), x_1 x_2, x_2 x_3, (x_3-1)x_1, x_4, \ldots, x_d\}$
forms a universal Gr\"{o}bner basis for $\mathcal{I}(\Xi_{d})$.
Hence, Theorem \ref{gctest2} implies that $\Xi_{d}$ has a unique associated monomial order quotient basis.

We now claim that $\Xi_{d}$ is non-Cartesian. To see this, we use induction on $d$.
When $d=3$, $\Xi_{3}$ is illustrated in (a) of Figure \ref{eg2}.
\begin{figure}[!htbp]
\centering
\subfigure[$\Xi_3$]{\includegraphics[width=6cm, height=6cm]{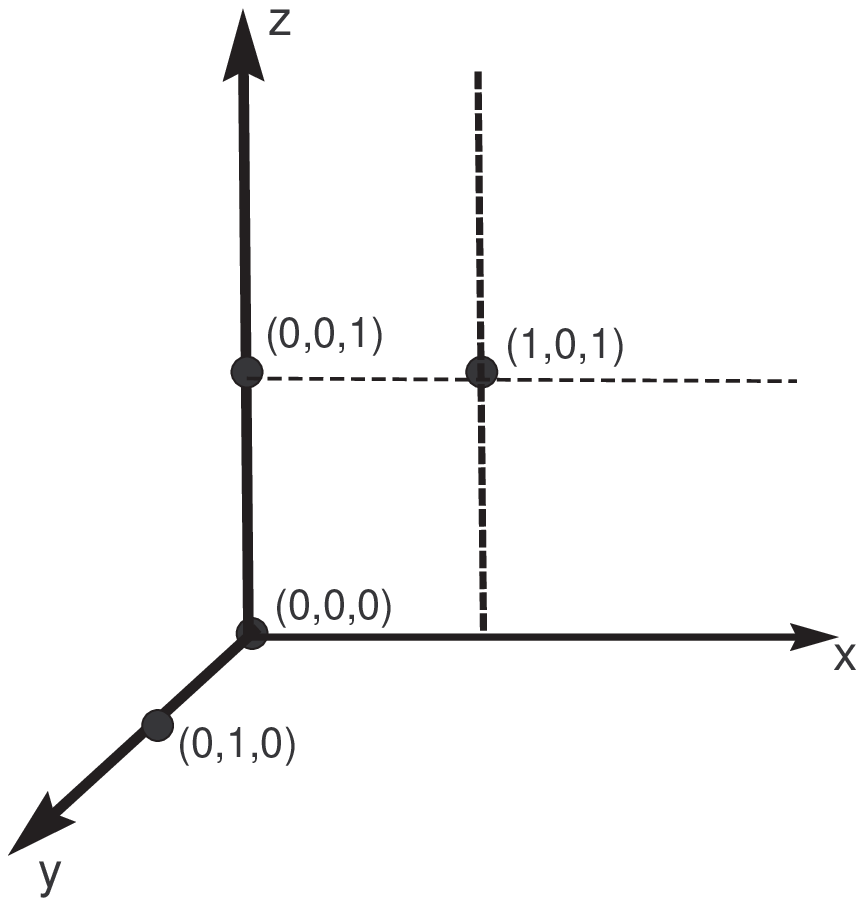}}
\subfigure[$\{1, x_3, x_2, x_1\}$]{\includegraphics[width=6cm, height=6cm]{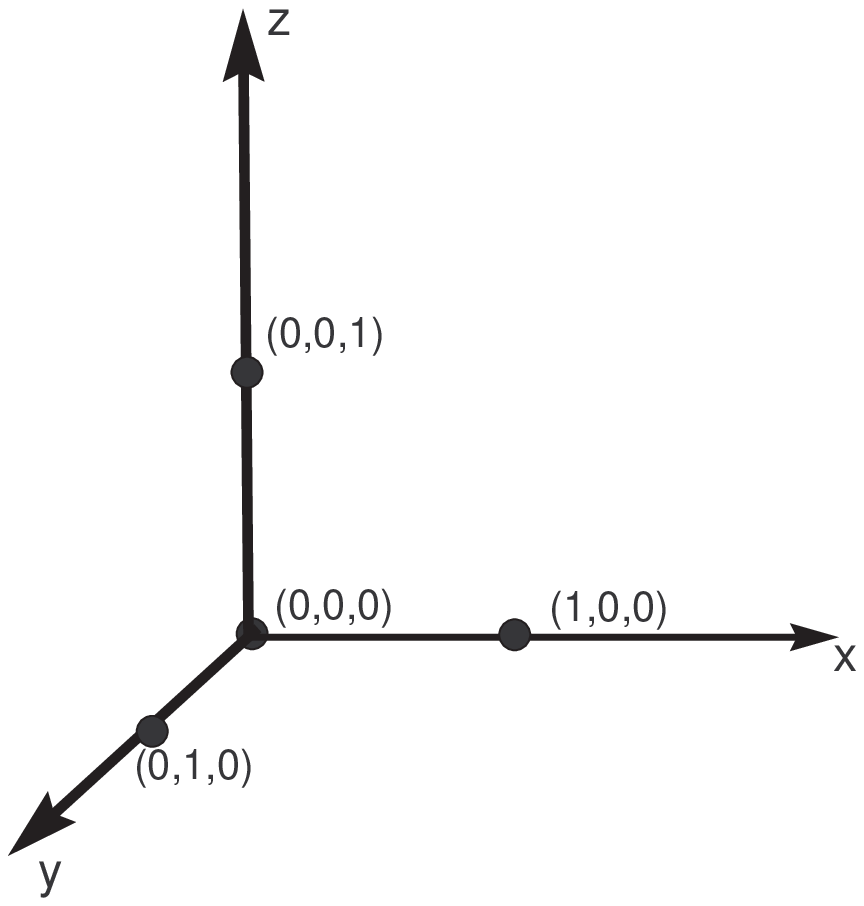}}
\caption{$\Xi_3$ and its unique associated monomial order quotient basis}\label{eg2}
\end{figure}

Recalling Theorem \ref{3dCartesiantest}, we can obtain
\begin{align*}
\mathscr{L}_{1,0}(\Xi_3)&=\{(0, 0), (0, 1), (1, 0)\}, &\mathscr{L}_{1,1}(\Xi)&=\{(0, 1)\}, \\
\mathscr{L}_{2,0}(\Xi_3)&=\{(0, 0), (0, 1), (1, 1)\}, &\mathscr{L}_{2,1}(\Xi)&=\{(0, 0)\}, \\
\mathscr{L}_{3,0}(\Xi_3)&=\{(0, 0), (0, 1)\}, &\mathscr{L}_{3,1}(\Xi)&=\{(0, 0), (1, 0)\}.
\end{align*}
We cover $\Xi_{3}$ with three sets of affine hyperplanes
$$
\{l_{1,0}, l_{1,1}\}, \{l_{2,0}, l_{2,1}\}, \{l_{3,0}, l_{3,1}\}
$$
where
$l_{1,0}=x_1-0, l_{1,1}=
x_1-1,l_{2,0}=x_2-0, l_{2,1}=
x_2-1,
l_{3,0}=x_3-0, l_{3,1}=x_3-1$.
Since
$\mathscr{L}_{3,0}(\Xi_{3})\not\supseteq\mathscr{L}_{3,1}(\Xi_3)$,
$\Xi_3$ is not Cartesian according to Theorem \ref{3dCartesiantest}.

Assume that $\Xi_{d-1}$ is non-Cartesian. If $\Xi_{d}$ is Cartesian, then Lemma \ref{miancar}
also leads to a contradiction to our
assumption.
\end{proof}

\section*{Acknowledgements}
The authors wish to thank the Reviewers for valuable suggestions and comments
that have improved the presentation of the paper.

\end{document}